\documentclass[12pt, reqno]{amsart}
\usepackage[margin = 1.3in]{geometry}   
\usepackage[english]{babel}
\usepackage{amssymb,amsmath,amsthm}
\usepackage{mathtools}
\usepackage{array}
\usepackage[shortlabels]{enumitem}
\RequirePackage{tikz}[2010/10/13] 
\usetikzlibrary{cd}

\usepackage{xspace}\xspaceaddexceptions{\,}

\makeatletter
\def\@listI{\leftmargin\leftmargini
    \parsep\parskip \itemsep -\parsep
    \itemsep 2pt}
\let\@listi\@listI
\makeatother

\theoremstyle{definition}
\newtheorem{definition}{Definition}[section]
\newtheorem{example}[definition]{Example}

\theoremstyle{plain}
\newtheorem{lemma}[definition]{Lemma}

\newtheorem{theorem}[definition]{Theorem}
\newtheorem{corollary}[definition]{Corollary}

\newcommand\A{{\mathbf A}}
\newcommand\B{{\mathbf B}}

\newcommand\I{{\mathbf I}}

\newcommand\V{{\mathcal{V}}}

\newcommand\PWK{\ensuremath{\mathrm{PWK}}\xspace}

\newcommand{\bit}{\begin{itemize}}    
\newcommand{\eit}{\end{itemize}}
\newcommand{\ben}{\begin{enumerate}}
\newcommand{\een}{\end{enumerate}}

\newcommand{\benroman}{\ben[\normalfont (i)]}  
\let\eroman\een

\newcommand{\bde}{\begin{description}}
\newcommand{\ede}{\end{description}}

\newcommand{\Var}{\mathnormal{V\mkern-.8\thinmuskip ar}}

\DeclareMathAlphabet{\mathbfsf}{\encodingdefault}{\sfdefault}{bx}n
\makeatletter
\providecommand*{\Dashv}{\mathrel{\mathpalette\@Dashv\vDash}}
\newcommand*{\@Dashv}[2]{\reflectbox{$\m@th#1#2$}}
\makeatother

\renewcommand\geq{\geqslant}

\newcommand\leqs{\leqslant}

\usepackage{units}

\newcommand\pair[1]{{\langle#1\rangle}}

\newcommand\PL{{\mathcal{P}_{l}}}
\newcommand\PLA{{\mathcal{P}_{l} (\mathbb{A})}}
\newcommand\PLB{{\mathcal{P}_{l} (\mathbb{B})}}

\newcommand\phih{{\varphi_{_h}}}
\useshorthands{"}
\defineshorthand{"-}{{}\nobreakdash-\hspace{0pt}}	

\usepackage[authormarkup=none]{changes}

\definechangesauthor[color=red]{Stef}

\usepackage[authormarkup=none]{changes}

\definechangesauthor[color=blue]{comm}

\begin{document}

\title[Dualities for Plonka sums]{Dualities for Plonka sums}

\author[S.Bonzio]{Stefano Bonzio}
\address{Stefano Bonzio, Universit\`a Politecnica delle Marche\\
Dipartimento di Scienze biomediche e sanit\`a pubblica\\
Ancona, Italy}
\email{stefano.bonzio@gmail.com}




\date{}

\begin{abstract}
P\l onka sums consist of an algebraic construction similar, in some sense to direct limits, which allows to represent classes of algebras defined by means of regular identities (namely those equations where the same set of variables appears on both sides). Recently, P\l onka sums have been connected to logic, as they provide algebraic semantics to logics obtained by imposing a syntactic filter to given logics.  
In this paper, I  present a very general topological duality for classes of algebras admitting a P\l onka sum representation in terms of dualisable algebras. 
\end{abstract}

\maketitle

\section{Introduction}

A formal identity $\varphi \thickapprox \psi$ is said to be \emph{regular} provided that exactly the same variables occur in the terms $\varphi$ and $\psi$. A variety $\V$ is called \emph{regular} whenever it satisfies identities which are regular only. The aim of this paper is showing a very simple way to construct topological dualities for regular varieties 
via the use of P\l onka sums.

On the other hand, a variety 
satisfying at least one identity which is not regular, is called \emph{irregular}. A relevant subclass of irregular variety is formed by the strongly irregular ones. A variety $\V$ is called \emph{strongly irregular} if it satisfies an identity of the kind $f(x,y)\approx x$, where $f(x, y)$ is any term of the language in which $x$ and $y$ really occur. 
Examples of strongly irregular varieties abound in logic, since every variety with a lattice reduct is irregular as witnessed by the term $f(x, y) \coloneqq x \land (x \lor y)$. 

The algebraic study of regular varieties traces back to the pioneering work of P\l onka \cite{Plo67}, who introduced a new class-operator $\PL(\cdot)$ nowadays called \textit{P\l onka sum}, and used it to prove that any regular variety $\V$ can be represented as P\l onka sum of a suitable strongly irregular variety $\V'$, in symbols $\PL(\V') = \V$. In this case $\V$ is called the \textit{regularization} of $\V'$, in a sense that we will be made precise.

Although the whole theory of P\l onka sum is purely algebraic, regular varieties have found applications in computer science, in particular in the theory of program semantics (see \cite{Libkin,Libkintechnical,Puhlmann}). Recently, P\l onka sums have been surprisingly connected to logic. Indeed, the algebraic semantics of one among the logics within the so-called Kleene family \cite{Kleene}, namely paraconsistent Weak Kleene logic, \PWK for short, coincide with the regularization of the variety of Boolean algebras, firstly axiomatised in \cite{Plonka69,plonka1984sum}.

\PWK has been essentially introduced by Halld\'en \cite{Hallden} and defended by Prior \cite{Prior} as a logic for handling reasonings that involve meaningless expressions and references to non-existing objects, respectively. The relation between \PWK and classical logic has been recently investigated in \cite{CC1+}, while proofs systems can be found in \cite{cc,Petrukhin}. The details of the connection between the logic \PWK and the regularization of Boolean algebras, also referred to as \emph{involutive bisemilattices}, are extensively studied in \cite{Bonzio16,tesiLuisa}. 

The link between logics and P\l onka sums can indeed be pushed further: the construction of P\l onka sums, originally devised for algebras only, can be extended also to logical matrices \cite{Bonzio18}, in such a way to provide algebraic semantics to the logics of variables inclusion. In detail, given a logic $\vdash$, a new consequence relation, denoted\footnote{The notation aims at stressing that the referred variable inclusion constraint goes from premises to conclusion, roughly speaking, from left to right.} by $\vdash^{l}$, can defined as follows: 
\[
\Gamma \vdash^{l} \varphi \Longleftrightarrow \text{ there is }\Delta \subseteq \Gamma \text{ s.t. }Var(\Delta) \subseteq Var(\varphi) \text{ and }\Delta \vdash \varphi.
\]

The models of the logic $\vdash^{l}$ are obtained out of matrix models of $\vdash$ via the construction of the P\l onka sum (see \cite{Bonzio18} for details). As a consequence, logics of variables inclusion embrace the class of logics often referred as \emph{infectious} (see \cite{Szmuc, Ferguson2015}), as they are semantically defined by a matrix containing a value that infects every operation in which it takes part (the logic \PWK is a prototypical example): P\l onka sums is indeed to most appropriate algebraic tool to express, algebraically, the notion of \emph{contamination}. Examples of logics of variables inclusion are also introduced in \cite{Ciuni1, Ciuni2}. In particular, they are applied for both modeling computer-programs affected by errors \cite{Ferguson} and in recent developments in the theory of truth \cite{Szmuctruth}.


On a different stream of research, the study of topological dualities for regular varieties traces back to the work of Gierz and Romanowska for distributive bisemilattices \cite{Romanowska}, the regularization of distributive lattices. The technique used there has been generalized a few years later to regular varieties in \cite{Romanowska96,Romanowska97} (a different approach can be found in \cite{Davey99}). 

We recently stated a slightly different duality, still based on P\l onka sums, for involutive bisemilattices, see \cite{nostraduality} (differences will be briefly explained in Section \ref{sec: dir,inv systems}). Dualities for (some) varieties of bisemilattices, although not relying on P\l onka sums, are considered in \cite{Ledda2017}. 

At the light of the above mentioned connection between logics (of variables inclusion) and P\l onka sums of (system of) algebras, the aim of this paper is provide a very general method for constructing topological dualities for algebras admitting a P\l onka sum representation in terms of dualisable algebraic structures (see Corollary \ref{cor: strong-dir-BA duale di strong-inv-SA}). 


The paper is structured as follows: Section \ref{sec: Preliminari} recalls the main results concerning the construction of P\l onka sums and their connection with regular varieties which will be used to implement our duality. Section \ref{sec: dir,inv systems} is devoted to introduce the categories used to build the duality, namely semilattice direct and inverse systems of an arbitrary category. Finally, Section \ref{sec: dualita} presents the main result.

\section{Preliminaries}\label{sec: Preliminari}

We start by providing all the necessary notions to construct P\l onka sums; then we will recall the connection with regular varieties.

For standard information on P\l onka sums we refer the reader to \cite{Plo67a, Plo67,Romanowska92, romanowska2002modes}. A \textit{semilattice} is an algebra $\A = \langle A, \lor\rangle$, where $\lor$ is a binary commutative, associative and idempotent operation. Given a semilattice $\A$ and $a, b \in A$, we set
\[
a \leq b \Longleftrightarrow a \lor b = b.
\]
It is easy to see that $\leq$ is a partial order on $A$.
\begin{definition}\label{Def:Directed-System-Matrices}
A \textit{semilattice direct system of algebras} consists in 
\benroman
\item a semilattice $I = \langle I, \lor\rangle$;
\item a family of algebras $\{ \A_{i} : i \in I \}$ with disjoint universes;
\item a homomorphism $f_{ij} \colon \A_{i} \to \A_{j}$, for every $i, j \in I$ such that $i \leq j$;
\eroman
moreover, $f_{ii}$ is the identity map for every $i \in I$, and if $i \leq j \leq k$, then $f_{ik} = f_{jk} \circ f_{ij}$.
\end{definition}

Let $X$ be a semilattice direct system of algebras as above. The \textit{P\l onka sum} over $X$, in symbols $\PL(X)$ or $\PL(\A_{i})_{i \in I}$, is the algebra defined as follows. The universe of $\PL(\A_{i})_{i \in I}$ is the union $\bigcup_{i \in I}A_{i}$. Moreover, for every $n$-ary basic operation $f$ (with $n\geq 1$)\footnote{In presence of nullary operations it is necessary to assume that $I$ also has a lower bound. See \cite{plonka1984sum} for details.}, and $a_{1}, \dots, a_{n} \in \bigcup_{i \in I}A_{i}$, we set
\[
f^{\PL(\A_{i})_{i \in I}}(a_{1}, \dots, a_{n}) \coloneqq f^{\A_{j}}(f_{i_{1} j}(a_{1}), \dots, f_{i_{n} j}(a_{n}))
\]
where $a_{1} \in A_{i_{1}}, \dots, a_{1} \in A_{i_{n}}$ and $j = i_{1} \lor \dots \lor i_{n}$.\


A simple example can be helpful to clarify the above definition. 

\begin{example}
Let $\A_i$ and $\A_j$ be isomorphic copies of the 4"-element Boolean algebra, with elements labelled as follows:
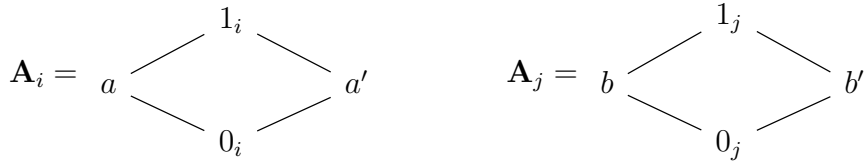
\[
\A_i = \begin{tikzcd}[row sep = tiny, arrows = {dash}]
 & 1_i & \\ a \arrow[ur, dash] & & a' \arrow[ul] \\ & 0_{i}\arrow[ul]\arrow[ur] &
  \end{tikzcd}
  \qquad  \qquad
  \A_j = \begin{tikzcd}[row sep = tiny, arrows = {dash}]
 & 1_j & \\ b \arrow[ur, dash] & & b' \arrow[ul] \\ & 0_j\arrow[ul]\arrow[ur] &
  \end{tikzcd}
\]
Let $\I$ be the linear order with two elements $i < j $, and $\mathbb{A}$ the semilattice direct system over $\I$ in which the homomorphism $p_{ij}\colon\A_{j }\to\A_j$ is given by $p_{ij}(a) = 1_j$ (and therefore $p_{ij}(a') = 0_j$). Hence the P\l onka sum $\PLA$ over this system is drawn in the following diagram (the arrow indicates the homomorphism $p_{ij}$):
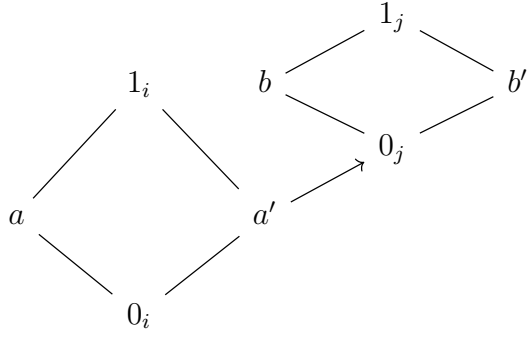
\vspace{15pt}
\vspace{-\baselineskip}
\[
\begin{tikzcd}[row sep = tiny, arrows = {dash}]
                &           &                                                                                    &               & 1_j &    \\
                &           & 1_i & b\arrow[ur] &  & b'\arrow[ul] \\
 & &        &               &                                                  0_j\arrow[ul]\arrow[ur] &    \\
                & a\arrow[uur] &  & a'\arrow[uul]\arrow[ur,->] &                                                         &    \\
                &           &   & &  &  \\
                &           &   & &  &  \\
                &           &   & &  &  \\
&  & 0_i\arrow[uuuur]\arrow[uuuul] &  &  &   \\ 
  \end{tikzcd}\]
 
 \noindent

We briefly sketch the way binary operations work in $\PLA$. For instance, 
\[
a\wedge^{\PL}a' = a\wedge^{\A_i}a' = 0_i
\]

More precisely, any operation involving two elements belonging to the same algebra is performed via the operations in such an algebra. On the other hand, 
 \[
 a'\wedge^{\PL}b = p_{ij}(a')\wedge^{\A_j}b = 0_j\wedge^{\A_j} b = 0_{j}.
 \]
 \qed
\end{example}

%

The theory of P\l onka sums is strictly related with a special kind of operation:

\begin{definition}\label{def: partition function}
Let $\A$ be an algebra of type $\nu$. A function $\cdot\colon A^2\to A$ is a \emph{partition function} in $\A$ if the following conditions are satisfied for all $a,b,c\in A$, $ a_1 , ..., a_n\in A^{n} $ and for any operation $g\in\nu$ of arity $n\geqslant 1$.
\begin{enumerate}
\item $a\cdot a = a$
\item $a\cdot (b\cdot c) = (a\cdot b) \cdot c $
\item $a\cdot (b\cdot c) = a\cdot (c\cdot b)$
\item $g(a_1,\dots,a_n)\cdot b = g(a_1\cdot b,\dots, a_n\cdot b)$
\item $b\cdot g(a_1,\dots,a_n) = b\cdot a_{1}\cdot_{\dots}\cdot a_n $
\end{enumerate}
\end{definition}

The next result makes explicit the relation between P\l onka sums and partition functions:

\begin{theorem}[\hbox{\cite{Plo67}}, Theorem II]\label{th: Teorema di Plonka}
Let $\A$ be an algebra of type $\nu$ with a partition funtion $\cdot$. The following conditions hold:
\begin{enumerate}
\item $A$ can be partitioned into $\{ A_{i} : i \in I \}$ where any two elements $a, b \in A$ belong to the same component $A_{i}$ exactly when
\[
a= a\cdot b \text{ and }b = b\cdot a.
\]
Moreover, every $A_{i}$ is the universe of a subalgebra $\A_{i}$ of $\A$.
\item The relation $\leq$ on $I$ given by the rule
\[
i \leq j \Longleftrightarrow \text{ there exist }a \in A_{i}, b \in A_{j} \text{ s.t. } b\cdot a =b
\]
is a partial order and $\langle I, \leq \rangle$ is a semilattice. 
\item For all $i,j\in I$ such that $i\leq j$ and $b \in A_{j}$, the map $f_{ij} \colon A_{i}\to A_{j}$, defined by the rule $f_{ij}(x)= x\cdot b$ is a homomorphism. The definition of $f_{ij}$ is independent from the choice of $b$, since $a\cdot b = a\cdot c$, for all $a\in A_i$ and $c\in A_j$.
\item $Y = \langle \langle I, \leq \rangle, \{ \A_{i} \}_{i \in I}, \{ f_{ij} \! : \! i \leq j \}\rangle$ is a direct system of algebras such that $\PL(Y)=\A$.
\end{enumerate}
\end{theorem}

The above result states that every algebra possessing a partition function can be associated to a semilattice system $\mathbb{A}$ and, most importantly,
the P\l onka sum over $\mathbb{A}$ is a representation of $\A$. 

The construction of P\l onka sums preserves the validity of the so-called \textit{regular identities} (see \cite[Theorem III]{Plo67}), i.e. identities of the form $ \varphi \approx \psi $ such that $\Var(\varphi) = \Var(\psi)$. In particular: 
\begin{theorem}[\mbox{\cite{Plo67}}, Theorem I]\label{lem:Plonka:sums:vs:regular:equations}
If $\mathbb{A}$ is a semilattice direct system of algebras containing at least two algebras, then in the algebra $\PLA$ all regular equations satisfied in all algebras of $\mathbb{A}$ are satisfied, whereas every other equations is false in $\PLA$.
\end{theorem}

A variety of algebras is called \emph{regular} if it does satisfies regular identities only. It is called irregular if it is not regular. In particular, an irregular variety $\V$ which possesses a term-definable operation $f(x,y)$ such that $\V\models f(x,y)\approx x$ is said to be \emph{strongly irregular}. Strongly irregular varieties are actually very common in mathematics: indeed, examples include the variety of groups and rings (as witnessed by the terms $f(x,y)\coloneqq x + (y - y)$, in additive notation) and any variety which has a lattice reduct (this includes, for instance, any variety of residuated lattices \cite{GaJiKoOn07}), as witnessed by the term $f(x,y)\coloneqq x\wedge (x\vee y)$.

Whenever $\V$ is an irregular variety, then we indicate by $R(\V)$, the \emph{regularization} of $\V$, namely the variety satisfying only the regular identities holding in $\V$.

The importance of regular and strongly irregular varieties, in the context of P\l onka sums, is resumed in the following:

\begin{theorem}[\hbox{\cite{Romanowska92}}, Theorem 7.1]\label{th: rappresentazione e strongly irregular varieties}
Let $\V$ be a strongly irregular variety. Then any element $\A\in R(\V) $ is isomorphic to the P\l onka sum over a direct system of algebras in $\V$. 
\end{theorem}

\section{The categories of semilattice systems}\label{sec: dir,inv systems}

The present section is meant to introduce the categories of direct and inverse semilattice systems which will be used to establish the main results (see Section \ref{sec: dualita}). 

We briefly recall the categories so as they are introduced in our previous work \cite{nostraduality}. Semilattice direct (and inverse) systems are, roughly speaking, obvious generalizations of direct (and inverse) systems in a given category, obtained by assuming the index set to be a semilattice instead of a (directed) pre-ordered set. These concepts find applications in several fields of mathematics (see for example \cite{shapebook}).

\begin{definition}\label{def: direct system}
 Let $ \mathfrak{C} $ be an arbitrary category. A \emph{semilattice direct system} in $ \mathfrak{C} $ is a triple $\mathbb{X}= \pair{X_{i}, p_{ii'}, I} $ such that 
 \begin{itemize}
\item[(i)] $ I $ is a join semilattice. 
\item[(ii)] $ X_i $ is an object in $ \mathfrak{C}$, for each $ i\in I $;
\item[(iii)] $p_{ii'} : X_{i} \to X_{i'}$ is a morphism of $ \mathfrak{C} $, for each pair $i\leqs i'$, satisfying that $p_{ii} $ is the identity in $ X_i $ and such that $ i\leq i'\leq i'' $ implies $ p_{i'i''} \circ p_{ii'}  = p_{ii''}$. 
\end{itemize}
\end{definition}


As a matter of convention, we indicate by $ \vee $ the semilattice operation on $ I $. 

Given two strongly direct systems $ \mathbb{X} $ and $ \mathbb{Y} $, a morphism is a pair $(\varphi, f_i):\mathbb{X}\to\mathbb{Y}$ such that:
\begin{itemize}
\item[i)] $ \varphi\colon I\rightarrow J $ is a semilattice homomorphism 
\item[ii)] $ f_i\colon X_{i}\rightarrow Y_{\varphi(i)} $ is a morphism of $ \mathfrak{C} $, making the diagram in Figure \ref{fig: diagramma strongly dir systems} commutative for each $ i, i'\in I $, $ i\leq i' $:
\end{itemize}
\begin{center}
\begin{figure}[h]
\begin{tikzpicture}
\draw (-4,0) node {$ Y_{\varphi(i)} $};
	\draw (4,0) node {$ Y_{\varphi(i')} $};
	
	\draw [line width=0.8pt, ->] (-3.5,0) -- (3.4,0);
	\draw (0,-0.4) node {\begin{footnotesize}$q_{\varphi(i)\varphi(i')}$\end{footnotesize}};
	
	\draw [line width=0.8pt, <-] (3.3,3) -- (-3.3,3);
	\draw (0, 3.3) node {\begin{footnotesize}$p_{ii'}$\end{footnotesize}};

	\draw (-4,3) node {$X_i$}; 
	
	\draw (4,3) node {$X_{i'}$};
		
	\draw [line width=0.8pt, ->] (-4,2.6) -- (-4,0.4);
	\draw (-4.6,1.7) node {\begin{footnotesize}$f_i$\end{footnotesize}};
	\draw (4.6,1.7) node {\begin{footnotesize}$f_{i'}$\end{footnotesize}};
	\draw [line width=0.8pt, <-] (4,0.4) -- (4,2.6);
	
\end{tikzpicture}
\caption{The commuting diagram defining morphisms of semilattice direct systems}\label{fig: diagramma strongly dir systems}
\end{figure}
\end{center}

Semilattice inverse systems, for an arbitrary category, are defined in an analogous, dual way. 
 
\begin{definition}\label{def: strong inverse system}
Let $ \mathfrak{C} $ be an arbitrary category, a \emph{semilattice inverse system} in the category $ \mathfrak{C} $ is a tern $\mathcal{X}= \pair{X_{i}, p_{ii'}, I} $ such that
\vspace{3pt} 
\begin{itemize}
\item[(i)] $ I $ is a join semilattice; 
\item[(ii)] for each $ i\in I $, $ X_i $ is an object in $ \mathfrak{C}$;
\item[(iii)] $p_{ii'} : X_{i'} \to X_i$ is a morphism of $ \mathfrak{C} $, for each pair $i\leqs i'$, satisfying that $p_{ii} $ is the identity in $ X_i $ and such that $ i\leq i'\leq i'' $ implies $ p_{ii'} \circ p_{i'i''}  = p_{ii''}$. 
\end{itemize}
\end{definition}


As already mentioned, the only difference making an inverse system a semilattice inverse system is the requirement on the index set to be a semilattice instead of a directed preorder. 

\begin{definition}\label{def: morfismi di strong inv}
Given two \emph{semilattice inverse systems} $ \mathcal{X}=\pair{X_{i}, p_{ii'}, I} $ and $ \mathcal{Y}=\pair{Y_{j}, q_{jj'}, J} $, a \emph{morphism} between $ \mathcal{X} $ and $\mathcal{Y} $ is a pair $  (\varphi, f_j) $ such that
\vspace{3pt}
\begin{itemize}
\item[i)]  $ \varphi :J\rightarrow I $ is a semilattice homomorphism;
\item[ii)] for each $ j\in J $, $ f_j: X_{\varphi(j)}\rightarrow Y_{j} $ is a morphism in $ \mathfrak{C} $, such that whenever $ j\leq j' $, then the diagram in Figure \ref{fig: diagramma strongly inv systems} commutes. 
\begin{center}
\begin{figure}[h]
\begin{tikzpicture}
\draw (-4,0) node {$ Y_{j} $};
	\draw (4,0) node {$ Y_{j'} $};
	
	\draw [line width=0.8pt, <-] (-3.6,0) -- (3.6,0);
	\draw (0,-0.4) node {\begin{footnotesize}$q_{jj'}$\end{footnotesize}};
	
	\draw [line width=0.8pt, ->] (3.3,3) -- (-3.3,3);
	\draw (0, 3.3) node {\begin{footnotesize}$p_{\varphi(j)\varphi(j')}$\end{footnotesize}};

	\draw (-4,3) node {$X_{\varphi(j)}$}; 
	
	\draw (4,3) node {$X_{\varphi(j')}$};
		
	\draw [line width=0.8pt, ->] (-4,2.6) -- (-4,0.4);
	\draw (-4.6,1.7) node {\begin{footnotesize}$f_j$\end{footnotesize}};
	\draw (4.6,1.7) node {\begin{footnotesize}$f_{j'}$\end{footnotesize}};
	\draw [line width=0.8pt, <-] (4,0.4) -- (4,2.6);
	
\end{tikzpicture}
\caption{The commuting diagram defining morphisms of semilattice inverse systems.}\label{fig: diagramma strongly inv systems}
\end{figure}
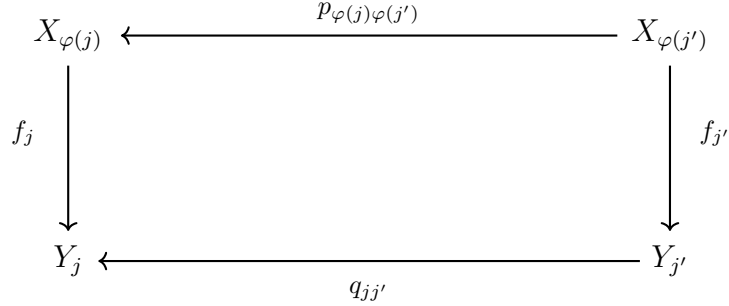
\end{center}
\end{itemize}
\end{definition}
Notice that, 
the assumption that $ \varphi\colon J\rightarrow I $ is a (semilattice) homomorphism implies that whenever $ j\leq j' $ then $ \varphi(j)\leq \varphi(j') $.

It is easily checked that, whenever $\mathfrak{C} $ an arbitrary category, then semilattice direct and inverse systems form two categories which we will refer to sem-dir-$\mathfrak{C}$ and sem-inv-$\mathfrak{C}$, respectively.


Dualities of arbitrary categories can be lifted to dualities of semilattice systems constructed via such duals categories. This result will be used in the next section:

\begin{theorem}[\mbox{\cite{nostraduality}}]\label{teorema: duality strong-dir strong-inv}
Let $\mathfrak{C}$ and $\mathfrak{D}$ be dually equivalent categories. Then \emph{sem-dir-}$\mathfrak{C}$ and \emph{sem-inv-}$\mathfrak{D}$ are dually equivalent categories.
\end{theorem}

A similar idea of lifting a duality has been ideated by Romanowska and Smith \cite{Romanowska96,Romanowska97}. 
In contrast with their approach, our duality is obtained constructing the categories sem-dir-$\mathcal{C}$ and sem-inv-$\mathcal{D}$ using the very same index set. On the other hand, they consider, on the algebraic side, the semilattice sum of an algebraic category and, on the topological, the semilattice representation of the dual spaces: the duality is then obtained by \emph{dualising} the semilattice of the index sets (the proof involves sophisticated categorical machinery).

\section{The Duality}\label{sec: dualita}

The notions of strongly irregular variety and regularization of a variety can be clearly defined as categories. We will say that $\mathfrak{C}$ is a strongly irregular algebraic category provided that its objects are strongly irregular varieties. In such case, $R(\mathfrak{C})$ is the algebraic category whose objects are regularizations of the objects in $\mathfrak{C}$. Moreover, we say that an algebraic category $\mathfrak{C}$ is \emph{dualisable} whenever it admits a dually equivalent topological category.

Theorem \ref{th: rappresentazione e strongly irregular varieties} states that, whenever $\mathfrak{C}$ is a strongly irregular category, the objects in  $R(\mathfrak{C})$ are isomorphic to the objects of the category sem-dir-$\mathfrak{C}$. We will show that they are also equivalent as categories. 

\begin{definition}\label{def: omomorfismo che preserva le fibre}
Given two semilattice direct systems $\mathbb{A} =\pair{\A_i, p_{ii'}, I}$ and $\mathbb{B}= \pair{\B_j, q_{jj'}, J}$ from an arbitrary category $\mathfrak{C}$ and an homomorphism $h\colon\PLA\to\PLB$, we say that $h$ \emph{preserves} the P\l onka fibres if, for every $i\in I$ there exists an index $j\in J$ such that $h(A_i)\subseteq B_j$. 
\end{definition}

We are interested in the following question: for which classes $\mathcal{K}$ of algebras, any homomorphism (between P\l onka sums of elements in $\mathcal{K}$) preserves the fibres? 

For the purpose of this paper, we confine our analysis to the case where $\mathcal{K}$ is a strongly irregular variety.

\begin{theorem}\label{th: omomorfismi tra somme di Plonka}
Let $ \mathbb{A} =\pair{\A_{i}, p_{ii'}, I} $ and $ \mathbb{B}=\pair{\B_{j}, q_{jj'},J} $ be semilattice direct systems of algebras, with $\{\A_i\}_{i\in I} $ and $\{\B_j\}_{j\in J} $ belonging to a variety $\V$, for each $i\in I$, $j\in J$ . Then any homomorphism $h:\PLA\to\PLB$ preserves the fibres if and only if $\V$ is a strongly irregular variety.
\end{theorem}
\begin{proof}
To simplify notation, set $\A=\PLA$ and $\B=\PLB$. 

($\Leftarrow$) Suppose $\mathcal{K}$ is a strongly irregular variety, i.e. it possesses a binary term definable operation $\circ$ such that $\V\models x\circ y \approx x$. Notice that $\circ$ defines a partition function on $\A$ and $\B$. Let $a_1$, $a_2\in A_i$ for some $i\in I$ and suppose, towards a contradiction, that $h(a_1)=b_1\in B_j$ and $h(a_2)= b_2\in B_k$ with $j\neq k\in J$. It follows that: 
\[
b_1 = h(a_1) = h(a_1\circ a_2) = h(a_1)\circ h(a_2) = b_1\circ b_2
\] 
and 
\[
b_2 = h(a_2) = h(a_2\circ a_1) = h(a_2)\circ h(a_1) = b_2\circ b_1,
\] 
which implies that the elements $b_1$, $b_2$ belong to the same algebra in $\mathbb{B}$, i.e. $j=k$, a contradiction.

($\Rightarrow$) Suppose $h$ preserves the fibres of the P\l onka sum, i.e. for each $i\in I$, there exists a $j\in J$ such that $h(A_i)\subseteq B_j$ and suppose, towards a contradiction that $\V$ is not strongly irregular. Since $\A$ and $\B$ are P\l onka sums of algebras in $\V$, there exists a binary term definable operation $f (x,y)$ which is a partition function on both $\A$ and $\B$. By Theorem \ref{th: Teorema di Plonka}, two elements $a,b\in \PLA$ belong to the same component $\A_i$ if and only if $f(a,b)= a$ and $f(b,a)=b$. Therefore, for each $i\in I$, $\A_i\models f(x,y)\approx x$. Moreover, for each $i\in I$, $\B_j\in \mathbb{H}(A_i)$, hence (since $\V$ is a variety) $\B_j\models f(x,y)\approx x$, for each $j\in J$. Then, since $\V$ is not strongly irregular, it follows that $\A_i $, $\B_j\not\in\V$, a contradiction.
\end{proof}


\begin{lemma}\label{lemma: morfismi tra somme di Plonka}
Let $ \mathbb{A} =\pair{\A_{i}, p_{ii'}, I} $ and $ \mathbb{B}=\pair{\B_{j}, q_{jj'},J} $ be semilattice direct systems of an arbitrary algebraic category $\mathfrak{C}$ and $(\varphi, f_{i})$ a morphism from $\mathbb{A}$ to $\mathbb{B}$. Then $h\colon \PLA\to\PLB$, defined as 
\[ 
h(a)\coloneqq f_{i}(a),
\]
where $i\in I$ is the index such that $a\in A_i$, is a morphism in $\mathfrak{C}$.
\end{lemma}
\begin{proof}
The map $h$ is well defined for every $i\in I$, as by assumption $f_i$ is morphism in $\mathfrak{C}$. Since $\mathfrak{C}$ is an algebraic category (where morphisms are homomorphisms of algebras), we only have to check that $h$ is compatible with all the operations of the P\l onka sum. To simplify notation, we set $\A=\PLA$, $\B=\PLB$, $a_1, ..., a_n\in\A$ with $i_1,...,i_n$ indexing the algebras to which they belong, $g$ a generic n-ary operation in the type of the considered algebras and, finally, $k= i_1\vee...\vee i_n$. Then,

\begin{align*}
h(g^{\A}(a_1,...,a_n)) &= h (g^{\A_k}(p_{{i_1}k}(a_1), ... , p_{{i_n}k}(a_n))) \\
&= f_{k}(g^{\A_k}(p_{{i_1}k}(a_1), ... , p_{{i_n}k}(a_n))) \\
&= g^{\B_{\varphi(k)}}(f_{k}(p_{{i_1}k}(a_1)), ... , f_{k}(p_{{i_n}k}(a_n))) \\
&= g^{\B_{\varphi(k)}}(q_{\varphi(i_1)\varphi(k)}(f_{i_1} (a_1)), ... , q_{\varphi(i_n)\varphi(k)}(f_{i_n} (a_n)) \\
&= g^{\B}(f_{i_1}(a_1), ... , f_{i_n}(a_n)) \\
&= g^{\B} (h(a_1), ... , h(a_n)),
\end{align*}
where the fourth equality is justified by the commutativity of the following diagram (which holds as, by assumption, $(\varphi, f_i)$ is morphism in sem-dir-$\mathfrak{C}$), for every $i\in\{ i_1 , ... , i_n\}$:

\begin{center}
\begin{tikzpicture}[scale=1]
\draw (-4,0) node {$ \B_{\varphi(i)} $};
	\draw (4,0) node {$ \B_{\varphi(k)} $};
	
	\draw [line width=0.8pt, ->] (-3.4,0) -- (3.3,0);
	\draw (0,-0.4) node {\begin{footnotesize}$q_{\varphi(i)\varphi(k)}$\end{footnotesize}};
	
	\draw [line width=0.8pt, <-] (3.3,3) -- (-3.3,3);
	\draw (0, 3.3) node {\begin{footnotesize}$p_{ik}$\end{footnotesize}};

	\draw (-4,3) node {$\A_i$}; 
	
	\draw (4,3) node {$\A_{k}$};
		
	\draw [line width=0.8pt, ->] (-4,2.6) -- (-4,0.4);
	\draw (-4.6,1.7) node {\begin{footnotesize}$f_i$\end{footnotesize}};
	\draw (4.6,1.7) node {\begin{footnotesize}$f_{k}$\end{footnotesize}};
	\draw [line width=0.8pt, <-] (4,0.4) -- (4,2.6);
	
\end{tikzpicture}
\end{center}
\end{proof}

\begin{lemma}\label{lem: omomorfismo degli indici}
Let $\V$ be a variety, $ \mathbb{A} =\pair{\A_{i}, p_{ii'}, I} $, $ \mathbb{B}=\pair{\B_{j}, q_{jj'},J} $ be semilattice direct systems of algebras in $\V$ 
and $ h\colon \PLA\to \PLB $ a homomorphism. Let $ \phih\colon I\rightarrow J $ be a map such that $ h(A_{i})\subseteq B_{\varphi_{_h}(i)} $. Then $ \phih $ is a semilattice homomorphism.
\end{lemma}
\begin{proof}
Let $ a_1, \dots, a_n\in \bigcup_{i\in I} A_i $, with $a_1\in A_{i_1},\dots, a_n\in A_{i_n}$ ($i_1,\dots i_n\in I$) and set $k=i_1\vee\dots\vee i_n$. We want to show that $\varphi_{_h}(k)= \varphi_{_h}(i_{1})\vee\dots\vee\varphi_{_h}(i_{n})$.

To simplify notation, set $\A=\PLA$ and $\B=\PLB$. Consider an arbitrary operation in the type of $\V$. Clearly, $h(f^{\A}(a_1,\dots,a_n))=f^{\B}(h(a_1),\dots,h(a_n)).$ By hypothesis, $h(a_1)\in B_{\varphi_{_h}(i_1)},\dots, h(a_n)\in B_{\varphi_{_h}(i_n)}$, therefore $f^{\B}(h(a_1),\dots,h(a_n))\in B_{j} $ with $j= \varphi(i_1)\vee\dots\vee \varphi(i_n)$. On the other hand, $f^{\A}(a_1,\dots a_n)\in A_k$, hence $h(f^{\A}(a_1,\dots a_n))\in B_{\varphi(k)}$. This shows that $\varphi_{_h}(k)= \varphi_{_h}(i_{1})\vee\dots\vee\varphi_{_h}(i_{n})$, i.e. is a semilattice homorphism. 

\end{proof}


\begin{theorem}\label{th: IBSL e strong-dir-BA sono equivalenti}
Let $\mathfrak{C}$ be a strongly irregular algebraic category. Then the categories $R(\mathfrak{C})$ and \emph{sem-dir}-$\mathfrak{C}$ are equivalent.
\end{theorem}

\begin{proof}

The equivalence is proved via the following functors: 
\[
\begin{tikzcd}[row sep = tiny, arrows = {dash}]
& & \mathcal{F} &  & & \\
& & & & & \\
& & & & & \\
  & R(\mathfrak{C})\arrow[rr, bend left = 25, rightarrow] &  & \text{sem-dir-}\mathfrak{C}\arrow[ll, bend left = 25, rightarrow] &                                  &  \\
  & & & & & \\
& & & & & \\
& & \mathcal{G} & & &   
  \end{tikzcd}
  \] 
 
Let $\A$ be an object in the category $R(\mathfrak{C})$. Since $\mathfrak{C}$ is strongly irregular, by Theorem \ref{th: rappresentazione e strongly irregular varieties}, we know that $\A\cong\PLA$, with $\mathbb{A}$ a semilattice direct system of algebras in $\mathfrak{C}$.  $\mathcal{F}$ associates to $\A$ the semilattice direct system $\mathbb{A}$. 

Consider a morphism in $R(\mathfrak{C})$, $h\colon\A\to \B $ and set $\A\cong\PLA$, $\B\cong\PLB$, with with $\mathbb{A}=\pair{\A_i, p_{ii'}, I}$ and $\mathbb{B}=\pair{\B_j, q_{jj'}, J}$ semilattice direct systems of algebras in $\mathfrak{C}$. Since $\mathfrak{C}$ is a strongly irregular variety, we know, by Theorem \ref{th: omomorfismi tra somme di Plonka}, that $h$ preserves the P\l onka fibres of the direct system $\mathbb{A}$ (arising from the P\l onka sum representation of $\A$), i.e. $h(A_i)\subseteq B_j$, for some $j\in J$. Hence, we can define a map $\varphi_{_h}\colon I\to J$ satisfying the assumptions of Lemma \ref{lem: omomorfismo degli indici}, which assures that $\varphi_{_h}$ is a semilattice homomorphism. Moreover, for each $i\in I$, the restriction of $h$ over $\A_i$, $h_{|A_{i}}$ is a homomorphism of algebras (objects) in $\mathfrak{C}$. $\mathcal{F}$ associate to the morphism $h$, the pair $(\varphi_{h}, h_{|A_{i}} ) $. 

Moreover, it is easily checked that the following diagram is commutative for each $i\leq i'$ (indeed $i\leq i'$ implies $\varphi_{h}(i)\leq\varphi_{h}(i')$)
\begin{center}
\begin{tikzpicture}[scale=1]
\draw (-4,0) node {$ \B_{\varphi_{h}(i)} $};
	\draw (4,0) node {$ \B_{\varphi_{h}(i')} $};
	
	\draw [line width=0.8pt, ->] (-3.3,0) -- (3.3,0);
	\draw (0,-0.4) node {\begin{footnotesize}$q_{\varphi_{h}(i)\varphi_{h}(i')}$\end{footnotesize}};
	
	\draw [line width=0.8pt, <-] (3.3,3) -- (-3.3,3);
	\draw (0, 3.3) node {\begin{footnotesize}$p_{ii'}$\end{footnotesize}};

	\draw (-4,3) node {$\A_i$}; 
	
	\draw (4,3) node {$\A_{i'}$};
		
	\draw [line width=0.8pt, ->] (-4,2.6) -- (-4,0.4);
	\draw (-4.6,1.7) node {\begin{footnotesize}$h_{|A_{i}}$\end{footnotesize}};
	\draw (4.6,1.7) node {\begin{footnotesize}$h_{|A_{i'}}$\end{footnotesize}};
	\draw [line width=0.8pt, <-] (4,0.4) -- (4,2.6);
	
\end{tikzpicture}
\end{center}

Therefore $\mathcal{F}(h)$ is a morphism from $\mathbb{A}$ to $\mathbb{B}$, showing that $\mathcal{F}$ is a covariant functor. 

On the other hand, $\mathcal{G}$ associates to an object $\mathbb{A}$ in the category sem-dir-$\mathfrak{C}$, the P\l onka sum $\PLA$ over $\A$, which is an object in $R(\mathfrak{C})$ (as $\mathfrak{C}$ is strongly irregular). Moreover, to each morphism $(\varphi, f_i)$, $\mathcal{G}$ associates the map $h\colon\PLA\to\PLB$, defined as $h(a)\coloneqq f_{i}(a)$, for each $a\in A_i$ and $i\in I$. 
Lemma \ref{lemma: morfismi tra somme di Plonka} assures that $h$ is indeed a morphism in $R(\mathfrak{C})$. 

It is easy to check that the compositions of the two functors are naturally isomorphic with the identities (in both categories). 
\end{proof}

As already mentioned, many algebraic structures arising in the study of logics are strongly irregular, since they possess a lattice reduct. Considering those ones admitting topological duals, the combination of Theorem \ref{th: IBSL e strong-dir-BA sono equivalenti} with Theorem \ref{teorema: duality strong-dir strong-inv} allows to construct the topological dual of the regularization of a variety.

\begin{corollary}\label{cor: strong-dir-BA duale di strong-inv-SA}
Let $\mathfrak{C}$ be a dualisable strongly irregular algebraic category with $\mathfrak{C}^{*}$ as topological dual. Then the categories $R(\mathfrak{C})$ and \emph{sem-inv}-$\mathfrak{C}^{*}$ are dually equivalent.
\end{corollary}


It is worthless to say that, to our's best knowledge, the construction of P\l onka sum has no analogous on the side of the topological representation spaces, so the class \emph{sem-inv}-$\mathfrak{C}^{*}$ remains basically a collection of spaces organized into a semilattice inverse system. A partial attempt to fill this gap is \cite{BonzioLoi}. 

A related question concerns the possibility of describing semilattice inverse systems of topological spaces as a unique space. This is done in some known special cases, as distributive bisemilattices \cite{Romanowska}, the P\l onka sum of distributive lattices and involutive bisemilattices \cite{nostraduality}, the P\l onka sum of Boolean algebras.


\vspace{30pt}

\begin{center}
\textbf{Acknowledgments}
\end{center}
\vspace{5pt}
\noindent
The work of the author is supported by the European Research Council, ERC Starting Grant GA:639276: ``Philosophy of Pharmacology: Safety, Statistical Standards, and Evidence Amalgamation''. The author expresses his gratitude to Andrea Loi and Anna Romanowska for the fruitful discussions on the topics of this paper.

\end{document}